\documentclass[10pt,twoside,12pt]{article}%
\usepackage{amssymb}
\usepackage{amsfonts}
\usepackage{amsmath}
\usepackage{graphicx}
\usepackage{latexsym}
\usepackage[a4paper,hmargin=2.5cm,vmargin=2.5cm]{geometry}%
\newtheorem{theorem}{Theorem}[section]

\newtheorem{lemma}{Lemma}[section]
\newtheorem{proposition}{Proposition}[section]
\newtheorem{remark}{Remark}[section]
\newtheorem{example}{Example}[section]

\newenvironment{proof}[1][Proof]{\noindent\textbf{#1.} }{\ \rule{0.5em}{0.5em}}
\begin{document}

\title{A bipolynomial fractional\\Dirichlet-Laplace problem}
\author{\textbf{{\large Dariusz Idczak}}\vspace{1mm}\\\textit{{\small Faculty of Mathematics and Computer Science}}\\\textit{{\small University of Lodz}}\\\textit{{\small 90-238 Lodz, Banacha 22}}\\\textit{{\small Poland}}\\{\small e-mail: idczak@math.uni.lodz.pl}}
\date{{\small \ }}
\maketitle

\begin{abstract}
In the paper, we derive an existence result for a nonlinear nonautonomous
partial elliptic system on an open bounded domain with Dirichlet boundary
conditions, containg fractional powers of \ the weak Dirichlet-Laplace
operator that are meant in the Stone-von Neumann operator calculus sense. We
apply a variational method which gives strong solutions of the problem under consideration.

\end{abstract}


\section{\textbf{Introduction}}

In the paper, we study strong solutions to the following general problem%
\begin{equation}%
%TCIMACRO{\dsum \nolimits_{i,j=0}^{k}}%
%BeginExpansion
{\displaystyle\sum\nolimits_{i,j=0}^{k}}
%EndExpansion
\alpha_{i}\alpha_{j}[(-\Delta)_{\omega}]^{\beta_{i}+\beta_{j}}u(x)=D_{u}%
F(x,u(x)),\ x\in\Omega\text{ a.e.}, \label{problemgeneral}%
\end{equation}
where $\alpha_{i}>0$ for $i=0,...,k$ ($k\in\mathbb{N}\cup\{0\}$) and
$0\leq\beta_{0}<\beta_{1}<...<\beta_{k}$, $[(-\Delta)_{\omega}]^{\gamma}$ is a
$\gamma$-power (in the sense of Stone-von Neumann operator calculus) of a
self-adjoint extension $(-\Delta)_{\omega}:D((-\Delta)_{\omega})\subset
L^{2}\rightarrow L^{2}$ (let us name it weak Dirichlet-Laplace operator) of
the Dirichlet-Laplace operator $(-\Delta):C_{c}^{\infty}\subset L^{2}%
\rightarrow L^{2}$, $\Omega\subset\mathbb{R}^{N}$ is a bounded open set,
$F:\Omega\times\mathbb{R}\rightarrow\mathbb{R}$, $D_{u}F$ is the partial
derivative of $F$ with respect to $u$, $C_{c}^{\infty}=C_{c}^{\infty}%
(\Omega,\mathbb{R})$, $L^{2}=L^{2}(\Omega,\mathbb{R})$ are real spaces of
smooth functions with compact support and square integrable functions, respectively.

Particular cases of the above problem are: the classical Dirichlet-Laplace
problem%
\begin{equation}
\lbrack(-\Delta)_{\omega}]u(x)=D_{u}F(x,u(x)),\ x\in\Omega\text{ a.e.},
\label{prob1}%
\end{equation}
the biharmonic equation%
\begin{equation}
\lbrack(-\Delta)_{\omega}]^{2}u(x)=D_{u}F(x,u(x)),\ x\in\Omega\text{ a.e.},
\label{prob2}%
\end{equation}
as well as the standard fractional problem%
\begin{equation}
\lbrack(-\Delta)_{\omega}]^{\beta}u(x)=D_{u}F(x,u(x)),\ x\in\Omega\text{
a.e.}, \label{prob3}%
\end{equation}

In recent years, fractional laplacians are extensively studied due to their
numerous applications. The authors use the different approaches to such
operators. Definition of the fractional Dirichlet-Laplacian adopted in our
paper comes from the Stone-von Neumann operator calculus and is based on the
spectral integral representation theorem for a self-adjoint operator in
Hilbert space. It reduces to a \textit{series}\ form which is taken by some
authors as a starting point (see \cite{Barrios}, \cite{Bors2}, \cite{CabreTan}%
) but our more general approach allows us to obtain useful properties of
fractional operators in a smart way.

The aim of our paper is to present the Stone-von Neumann approach and to
extend a variational method to bipolynomial problem (\ref{problemgeneral}).
Such a method was applied by other authors (see \cite{Barrios}, \cite{Bors2})
but to problems containing a single fractional Dirichlet-Laplacian. An
important issue of our study is an equivalence of the solutions obtained with
the aid of this method and the strong ones, that, to the best of our
knowledge, was not investigated by other authors.

The paper consists of two parts. In the first part, we give some basics from
the spectral theory of self-adjoint operators in real Hilbert space (Stone-von
Neumann operator calculus), recall the Friedrich's procedure of construction
of a self-adjoint extension of the symmetric operator and apply this procedure
to the Dirichlet-Laplace operator. In the second part, we investigate the
powers of the weak Dirichlet-Laplace operator (including their domains), a
connection between weak and strong solutions of a bipolynomial equation with a
self-adjoint operator and apply an extension of the mentioned variational
method to problem (\ref{problemgeneral}).

\section{Preliminaries}

\subsection{Self-adjoint operators in Hilbert space}

This subsection contains the results from the theory of self-adjoint operators
in real Hilbert space. Results presented in this section comes from
\cite{Alex}, \cite{Mlak} where they are derived in the case of complex Hilbert
space but their proofs can be moved without any or with small changes to the
case of real Hilbert space.

Let $H$ be a real Hilbert space with a scalar product $\left\langle
\cdot,\cdot\right\rangle :H\times H\rightarrow\mathbb{R}$. Denote by $\Pi(H)$
the set of all projections of $H$ on closed linear subspaces and by
$\mathcal{B}$ - the $\sigma$-algebra of Borel subsets of $\mathbb{R}$. By the
spectral measure in $\mathbb{R}$ we mean a set function $E:\mathcal{B}%
\rightarrow\Pi(H)$ that satisfies the following conditions:

\begin{itemize}
\item[$\cdot$] for any $x\in H$, the function%
\begin{equation}
\mathcal{B}\ni P\longmapsto E(P)x\in H \label{vm}%
\end{equation}
is a vector measure

\item[$\cdot$] $E(\mathbb{R})=I$ (identity operator)

\item[$\cdot$] $E(P\cap Q)=E(P)\circ E(Q)$ for $P,Q\in\mathcal{B}$.
\end{itemize}

\noindent By a support of a spectral measure $E$ we mean the complement of the
sum of all open subsets of $\mathbb{R}$ with zero spectral measure.

If $b:\mathbb{R}\rightarrow\mathbb{R}$ is a bounded Borel measurable function,
defined $E$ - a.e., then the integral $%
%TCIMACRO{\dint \nolimits_{-\infty}^{\infty}}%
%BeginExpansion
{\displaystyle\int\nolimits_{-\infty}^{\infty}}
%EndExpansion
b(\lambda)E(d\lambda)$ is defined by%
\[
(%
%TCIMACRO{\dint \nolimits_{-\infty}^{\infty}}%
%BeginExpansion
{\displaystyle\int\nolimits_{-\infty}^{\infty}}
%EndExpansion
b(\lambda)E(d\lambda))x=%
%TCIMACRO{\dint \nolimits_{-\infty}^{\infty}}%
%BeginExpansion
{\displaystyle\int\nolimits_{-\infty}^{\infty}}
%EndExpansion
b(\lambda)E(d\lambda)x
\]
for any $x\in H$ where the integral $%
%TCIMACRO{\dint \nolimits_{-\infty}^{\infty}}%
%BeginExpansion
{\displaystyle\int\nolimits_{-\infty}^{\infty}}
%EndExpansion
b(\lambda)E(d\lambda)x$ (with respect to the vector measure (\ref{vm})) is
defined in a standard way, with the aid of the sequence of simple functions
converging $E(d\lambda)x$ - a.e. to $b$ (see \cite{Alex}).

If $b:\mathbb{R}\rightarrow\mathbb{R}$ is an unbounded Borel measurable
function, defined $E$ - a.e., then, for any $x\in H$ such that
\begin{equation}%
%TCIMACRO{\dint \nolimits_{-\infty}^{\infty}}%
%BeginExpansion
{\displaystyle\int\nolimits_{-\infty}^{\infty}}
%EndExpansion
\left\vert b(\lambda)\right\vert ^{2}\left\Vert E(d\lambda)x\right\Vert
^{2}<\infty\label{fm}%
\end{equation}
(the above integral is taken with respect to the nonnegative measure
$\mathcal{B\ni}P\longmapsto\left\Vert E(P)x\right\Vert ^{2}\in\mathbb{R}%
_{0}^{+}$), there exists the limit%
\[
\lim%
%TCIMACRO{\dint \nolimits_{-\infty}^{\infty}}%
%BeginExpansion
{\displaystyle\int\nolimits_{-\infty}^{\infty}}
%EndExpansion
b_{n}(\lambda)E(d\lambda)x
\]
of integrals where
\[
b_{n}:\mathbb{R\ni\lambda\longmapsto}\left\{
\begin{array}
[c]{ccc}%
b(\lambda) & \text{as} & \left\vert b(\lambda)\right\vert \leq n\\
0 & \text{as} & \left\vert b(\lambda)\right\vert >n
\end{array}
\right.
\]
for $n\in\mathbb{N}$. Let us denote the set of all points $x$ with property
(\ref{fm}) by $D$. One proves that $D$ is a dense linear subspace of $H$ and
by $%
%TCIMACRO{\dint \nolimits_{-\infty}^{\infty}}%
%BeginExpansion
{\displaystyle\int\nolimits_{-\infty}^{\infty}}
%EndExpansion
b(\lambda)E(d\lambda)$ one denotes the operator%
\[%
%TCIMACRO{\dint \nolimits_{-\infty}^{\infty}}%
%BeginExpansion
{\displaystyle\int\nolimits_{-\infty}^{\infty}}
%EndExpansion
b(\lambda)E(d\lambda):D\subset H\rightarrow H
\]
given by%
\[
(%
%TCIMACRO{\dint \nolimits_{-\infty}^{\infty}}%
%BeginExpansion
{\displaystyle\int\nolimits_{-\infty}^{\infty}}
%EndExpansion
b(\lambda)E(d\lambda))x=\lim%
%TCIMACRO{\dint \nolimits_{-\infty}^{\infty}}%
%BeginExpansion
{\displaystyle\int\nolimits_{-\infty}^{\infty}}
%EndExpansion
b_{n}(\lambda)E(d\lambda)x.
\]
Of course, $D=H$ and
\[
\lim%
%TCIMACRO{\dint \nolimits_{-\infty}^{\infty}}%
%BeginExpansion
{\displaystyle\int\nolimits_{-\infty}^{\infty}}
%EndExpansion
b_{n}(\lambda)E(d\lambda)x=%
%TCIMACRO{\dint \nolimits_{-\infty}^{\infty}}%
%BeginExpansion
{\displaystyle\int\nolimits_{-\infty}^{\infty}}
%EndExpansion
b(\lambda)E(d\lambda)x
\]
when $b:\mathbb{R}\rightarrow\mathbb{R}$ is a bounded Borel measurable
function, defined $E$ - a.e.

For $x\in D$, we have%
\[
\left\Vert (%
%TCIMACRO{\dint \nolimits_{-\infty}^{\infty}}%
%BeginExpansion
{\displaystyle\int\nolimits_{-\infty}^{\infty}}
%EndExpansion
b(\lambda)E(d\lambda))x\right\Vert ^{2}=%
%TCIMACRO{\dint \nolimits_{-\infty}^{\infty}}%
%BeginExpansion
{\displaystyle\int\nolimits_{-\infty}^{\infty}}
%EndExpansion
\left\vert b(\lambda)\right\vert ^{2}\left\Vert E(d\lambda)x\right\Vert ^{2}.
\]
Moreover,%
\begin{equation}
(%
%TCIMACRO{\dint \nolimits_{-\infty}^{\infty}}%
%BeginExpansion
{\displaystyle\int\nolimits_{-\infty}^{\infty}}
%EndExpansion
b(\lambda)E(d\lambda))^{\ast}=%
%TCIMACRO{\dint \nolimits_{-\infty}^{\infty}}%
%BeginExpansion
{\displaystyle\int\nolimits_{-\infty}^{\infty}}
%EndExpansion
b(\lambda)E(d\lambda), \label{samosp}%
\end{equation}
i.e. the operator $%
%TCIMACRO{\dint \nolimits_{-\infty}^{\infty}}%
%BeginExpansion
{\displaystyle\int\nolimits_{-\infty}^{\infty}}
%EndExpansion
b(\lambda)E(d\lambda)$ is self-adjoint.

\begin{remark}
\label{sense}To integrate a Borel measurable function $b:B\rightarrow
\mathbb{R}$ where $B$ is a Borel set containing the support of the measure
$E$, it is sufficient to extend $b$ on $\mathbb{R}$ to a whichever Borel
measurable function (putting, for example, $b(\lambda)=0$ for $\lambda\notin
B$).
\end{remark}

If $b:\mathbb{R}\rightarrow\mathbb{R}$ is Borel measurable and $\sigma
\in\mathcal{B}$, then by the integral
\[%
%TCIMACRO{\dint \nolimits_{\sigma}}%
%BeginExpansion
{\displaystyle\int\nolimits_{\sigma}}
%EndExpansion
b(\lambda)E(d\lambda)
\]
we mean the integral
\[%
%TCIMACRO{\dint \nolimits_{-\infty}^{\infty}}%
%BeginExpansion
{\displaystyle\int\nolimits_{-\infty}^{\infty}}
%EndExpansion
\chi_{\sigma}(\lambda)b(\lambda)E(d\lambda)
\]
where $\chi_{\sigma}$ is the characteristic function of the set $\sigma$
(integral $%
%TCIMACRO{\dint \nolimits_{\sigma}}%
%BeginExpansion
{\displaystyle\int\nolimits_{\sigma}}
%EndExpansion
b(\lambda)E(d\lambda)$ can be also defined with the aid of the restriction of
$E$ to the set $\sigma$).

Next theorem plays the fundamental role in the spectral theory of self-adjoint
operators (below, $\sigma(A)$ denotes the spectrum of an operator
$A:D(A)\subset H\rightarrow H$).

\begin{theorem}
\label{main}If $A:D(A)\subset H\rightarrow H$ is self-adjoint and the
resolvent set $\rho(A)$ is non-empty, then there exists a unique spectral
measure $E$ with the closed support $\Lambda=\sigma(A)$, such that%
\[
A=%
%TCIMACRO{\dint \nolimits_{-\infty}^{\infty}}%
%BeginExpansion
{\displaystyle\int\nolimits_{-\infty}^{\infty}}
%EndExpansion
\lambda E(d\lambda)=%
%TCIMACRO{\dint \nolimits_{\sigma(A)}}%
%BeginExpansion
{\displaystyle\int\nolimits_{\sigma(A)}}
%EndExpansion
\lambda E(d\lambda).
\]

\end{theorem}

The basic notion in the Stone-von Neumann operator calculus is a function of a
self-adjoint operator. Namely, if $A:D(A)\subset H\rightarrow H$ is
self-adjoint and $E$ is the spectral measure determined according to the above
theorem, then, for any Borel measurable function $b:\mathbb{R}\rightarrow
\mathbb{R}$, one defines the operator $b(A)$ by%
\[
b(A)=%
%TCIMACRO{\dint \nolimits_{-\infty}^{\infty}}%
%BeginExpansion
{\displaystyle\int\nolimits_{-\infty}^{\infty}}
%EndExpansion
b(\lambda)E(d\lambda)=%
%TCIMACRO{\dint \nolimits_{\sigma(A)}}%
%BeginExpansion
{\displaystyle\int\nolimits_{\sigma(A)}}
%EndExpansion
b(\lambda)E(d\lambda).
\]
It is known that the spectrum $\sigma(b(A))$ of $b(A)$ is given by%
\begin{equation}
\sigma(b(A))=\overline{b(\sigma(A))} \label{ow}%
\end{equation}
provided that $b$ is continuous (it is sufficient to assume that $b$ is
continuous on $\sigma(A)$).

We have the following general results.

\begin{proposition}
\label{sum}If $b,d:\mathbb{R}\rightarrow\mathbb{R}$ are Borel measurable
functions and $E$ is the spectral measure for a self-adjoint operator
$A:D(A)\subset H\rightarrow H$ with non-empty resolvent set, then%
\[
(b+d)(A)\supset b(A)+d(A)
\]
and%
\[
(b+d)(A)=b(A)+d(A)
\]
if and only if%
\[
D((b+d)(A))\subset D(d(A))
\]
or, if and only if%
\[
D((b+d)(A))\subset D(b(A)).
\]

\end{proposition}

\begin{proposition}
\label{zloz}If $b,d:\mathbb{R}\rightarrow\mathbb{R}$ are Borel measurable
functions and $E$ is the spectral measure for a self-adjoint operator
$A:D(A)\subset H\rightarrow H$ with non-empty resolvent set, then
\[
(b\cdot d)(A)\supset b(A)\circ d(A)
\]
and%
\begin{equation}
(b\cdot d)(A)=b(A)\circ d(A) \label{rowbd}%
\end{equation}
if and only if%
\[
D((b\cdot d)(A))\subset D(d(A)).
\]

\end{proposition}

Using the above propositions one can deduce

\begin{proposition}
If $E$ is the spectral measure for a self-adjoint operator $A:D(A)\subset
H\rightarrow H$ with non-empty resolvent set, then for any $n\in\mathbb{N}$,
$n\geq2$, and a Borel measurable function $b:\mathbb{R}\rightarrow\mathbb{R}$,
one has%
\begin{equation}
(b(A))^{n}=b^{n}(A). \label{formulan}%
\end{equation}

\end{proposition}

When $b(\lambda)=\lambda$, (\ref{formulan}) gives%
\begin{equation}
A^{n}=%
%TCIMACRO{\dint \nolimits_{-\infty}^{\infty}}%
%BeginExpansion
{\displaystyle\int\nolimits_{-\infty}^{\infty}}
%EndExpansion
\lambda^{n}E(d\lambda) \label{formulann}%
\end{equation}
(if $n=1$, then (\ref{formulann}) follows from Theorem \ref{main}). Since
$E(\mathbb{R})=I$, therefore the identity operator $I$ can be written as%
\[
I=%
%TCIMACRO{\dint \nolimits_{-\infty}^{\infty}}%
%BeginExpansion
{\displaystyle\int\nolimits_{-\infty}^{\infty}}
%EndExpansion
1E(d\lambda).
\]
More generally,

\begin{proposition}
\label{wielomn}If $E$ is the spectral measure for a self-adjoint operator
$A:D(A)\subset H\rightarrow H$ with non-empty resolvent set, then%
\[
\alpha_{k}A^{k}+...+\alpha_{1}A+\alpha_{0}I=%
%TCIMACRO{\dint \nolimits_{-\infty}^{\infty}}%
%BeginExpansion
{\displaystyle\int\nolimits_{-\infty}^{\infty}}
%EndExpansion
(\alpha_{k}\lambda^{k}+...+\alpha_{1}\lambda^{1}+\alpha_{0})E(d\lambda).
\]

\end{proposition}

Let $\beta>0$ and $\sigma(A)\subset\lbrack0,\infty)$. According to the Remark
\ref{sense} by $A^{\beta}$ we mean the operator%
\[
A^{\beta}=%
%TCIMACRO{\dint \nolimits_{-\infty}^{\infty}}%
%BeginExpansion
{\displaystyle\int\nolimits_{-\infty}^{\infty}}
%EndExpansion
b(\lambda)E(d\lambda)
\]
where
\[
b:\mathbb{R}\ni\lambda\rightarrow\left\{
\begin{array}
[c]{ccc}%
0 & ; & \lambda<0\\
\lambda^{\beta} & ; & \lambda\geq0
\end{array}
\right.  .
\]
Using Proposition \ref{zloz} one can prove

\begin{proposition}
\label{potpot}If $E$ is the spectral measure for a self-adjoint operator
$A:D(A)\subset H\rightarrow H$ with $\sigma(A)\subset\lbrack0,\infty)$, then%
\begin{equation}
A^{\beta_{2}}\circ A^{\beta_{1}}=A^{\beta_{2}+\beta_{1}} \label{beta1beta2}%
\end{equation}
for $\beta_{2}$, $\beta_{1}>0$.
\end{proposition}

\begin{proposition}
\label{wielomf}If $E$ is the spectral measure for a self-adjoint operator
$A:D(A)\subset H\rightarrow H$ with $\sigma(A)\subset\lbrack0,\infty)$, then%
\[
\alpha_{k}A^{\beta_{k}}+...+\alpha_{1}A^{\beta_{1}}+\alpha_{0}A^{\beta_{0}}=%
%TCIMACRO{\dint \nolimits_{-\infty}^{\infty}}%
%BeginExpansion
{\displaystyle\int\nolimits_{-\infty}^{\infty}}
%EndExpansion
w(\lambda)E(d\lambda).
\]
where%
\[
w:\mathbb{R}\ni\lambda\rightarrow\left\{
\begin{array}
[c]{ccc}%
0 & ; & \lambda<0\\
\alpha_{k}\lambda^{\beta_{k}}+...+\alpha_{1}\lambda^{\beta_{1}}+\alpha
_{0}\lambda^{\beta_{0}} & ; & \lambda\geq0
\end{array}
\right.
\]
and $0\leq\beta_{0}<\beta_{1}<...<\beta_{k}.$
\end{proposition}

In the next we shall also use

\begin{proposition}
\label{odwrot}If $E$ is the spectral measure for a self-adjoint operator
$A:D(A)\subset H\rightarrow H$ with non-empty resolvent set and $b:\mathbb{R}%
\rightarrow\mathbb{R}$ is a Borel measurable function such that $b(\lambda
)\neq0$ a.e. with respect to $E$, then there exists the inverse operator
$[b(A)]^{-1}$ and%
\[
\lbrack b(A)]^{-1}=%
%TCIMACRO{\dint \nolimits_{-\infty}^{\infty}}%
%BeginExpansion
{\displaystyle\int\nolimits_{-\infty}^{\infty}}
%EndExpansion
\frac{1}{b(\lambda)}E(d\lambda).
\]

\end{proposition}

\subsection{Friedrich's extension}

In this subsection we recall the Friedrich's procedure of extension of a
symmetric operator to a self-adjoint one (see \cite{Helffer}) and apply this
procedure in details to the Dirichlet-Laplace operator.

\subsubsection{General case}

Let $H$ be a real Hilbert space with a scalar product $\left\langle
\cdot,\cdot\right\rangle $ and $T_{0}:D(T_{0})\subset H\rightarrow H$ - a
densely defined symmetric and positive-definite operator, i.e. $\left\langle
T_{0}u,u\right\rangle \geq\left\Vert u\right\Vert ^{2}$ for $u\in D(T_{0})$
(it is sufficient to assume that $T_{0}$ is semi-bounded, i.e. $\left\langle
T_{0}u,u\right\rangle \geq-\alpha\left\Vert u\right\Vert ^{2}$ for $u\in
D(T_{0})$, where $\alpha>0$). Denote%
\[
a_{0}(u,v)=\left\langle T_{0}u,v\right\rangle
\]
for $u,v\in D(T_{0})$ and define a set $V$ in the following way: $u\in V$ if
and only if there exists a sequence $(u_{n})\subset D(T_{0})$ which is
fundamental in $D(T_{0})$ with respect to the norm $p_{0}(u)=\sqrt{a_{0}%
(u,u)}$ and $u=\lim u_{n}$ in $H$. One proves that $V$ is dense in $H$. Let us
introduce in $V$ the scalar product%
\[
\left\langle u,v\right\rangle _{V}=\lim a_{0}(u_{n},v_{n})
\]
where $(u_{n}),(v_{n})$ are sequences from the definition $V$ (one can show
that the value $\left\langle u,v\right\rangle _{V}$ does not depend on the
choice of sequences $(u_{n}),(v_{n})$). $V$ with the above scalar product is
complete. Moreover, $\left\Vert u\right\Vert _{V}\geq\left\Vert u\right\Vert
_{H}$ for any $u\in V$. Now, let us define the set%
\begin{align*}
D(S)  &  =\{u\in V;\ \left\langle u,\cdot\right\rangle _{V}\text{ is a
continuous functional on }V\text{ }\\
&  \text{with respect to topology induced from }H\}\text{.}%
\end{align*}
If $u\in D(S)$, then by $Su$ we denote a unique element from $H$ determining
(according to the Riesz theorem) the functional $\left\langle u,\cdot
\right\rangle _{V}$ (after its continuous extension on the whole space $H$).
Using a general version of Lax-Milgram theorem due to Lions (see
\cite[Theorems 3.6 and 3.7]{Helffer}) we obtain the following result.

\begin{theorem}
\label{Fried}Operator $S:D(S)\subset H\rightarrow H$ is densely defined
bijection and self-adjoint extension of $T_{0}$, i.e. $S=S^{\ast}$ and
$T_{0}\subset S$.
\end{theorem}

\subsubsection{Case of Dirichlet-Laplace operator}

Let $\Omega\subset\mathbb{R}^{N}$ be an open bounded set and
\[
T_{0}=-\Delta:D(T_{0})=C_{c}^{\infty}\subset L^{2}\rightarrow L^{2}%
\]
be the classical Dirichlet-Laplace operator. Of course, $T_{0}$ is symmetric.
From the Poincare inequality (see \cite[Corollary 9.19]{Brezis}) (\footnote{To
use this inequality it is sufficient to assume that $\Omega$ has finite
measure or a bounded projection on some axis (cf. \cite[Chapter 9.4, Remark
21]{Brezis})}) it follows that $T_{0}$ (up to a constant factor) is
positive-definite. Define%
\[
a_{0}(u,v)=%
%TCIMACRO{\dint \nolimits_{\Omega}}%
%BeginExpansion
{\displaystyle\int\nolimits_{\Omega}}
%EndExpansion
(-\Delta)u(x)v(x)dx=%
%TCIMACRO{\dint \nolimits_{\Omega}}%
%BeginExpansion
{\displaystyle\int\nolimits_{\Omega}}
%EndExpansion
\nabla u(x)\nabla v(x)dx
\]
for $u,v\in C_{c}^{\infty}$. In an elementary way, we obtain (below,
$H_{0}^{1}=H_{0}^{1}(\Omega,\mathbb{R})$ is the closure of $C_{c}^{\infty}$ in
$H^{1}=H^{1}(\Omega,\mathbb{R})$, where $H^{1}(\Omega,\mathbb{R}%
)=W^{1,2}(\Omega,\mathbb{R})$ is the classical Sobolev space)

\begin{lemma}
The set $V$ coincides with $H_{0}^{1}$ and%
\[
\left\langle u,v\right\rangle _{V}=%
%TCIMACRO{\dint \nolimits_{\Omega}}%
%BeginExpansion
{\displaystyle\int\nolimits_{\Omega}}
%EndExpansion
\nabla u(x)\nabla v(x)dx.
\]

\end{lemma}

\begin{proof}
Let $u\in V$, i.e. there exists a sequence $(\varphi_{n})\subset C_{c}%
^{\infty}$ and a function $g\in L^{2}$ such that $\varphi_{n}\rightarrow u$ in
$L^{2}$ and $\nabla\varphi_{n}\rightarrow g$ in $L^{2}$. Consequently (see
\cite[Chapter 9.1, Remark 4]{Brezis}), $u\in H^{1}$ and $g=\nabla u$. In the
other words, $\varphi_{n}\rightarrow u$ in $H^{1}$. This means that $u\in
H_{0}^{1}$. Inclusion $H_{0}^{1}\subset V$ is obvious.

We also have%
\[
\left\langle u,v\right\rangle _{V}=\lim a_{0}(\varphi_{n},\psi_{n})=\lim%
%TCIMACRO{\dint \nolimits_{\Omega}}%
%BeginExpansion
{\displaystyle\int\nolimits_{\Omega}}
%EndExpansion
\nabla\varphi_{n}(x)\nabla\psi_{n}(x)dx=%
%TCIMACRO{\dint \nolimits_{\Omega}}%
%BeginExpansion
{\displaystyle\int\nolimits_{\Omega}}
%EndExpansion
\nabla u(x)\nabla v(x)dx
\]
where $(\varphi_{n})$, $(\psi_{n})\subset C_{c}^{\infty}$ are the sequences
from the definition of $V$, corresponding to $u$, $v$, respectively.
\end{proof}

We shall say that $u:\Omega\rightarrow\mathbb{R}$ has a weak (minus)
Dirichlet-Laplacian if $u\in H_{0}^{1}$ and there exists a function $g\in
L^{2}$ such that%
\[%
%TCIMACRO{\dint \nolimits_{\Omega}}%
%BeginExpansion
{\displaystyle\int\nolimits_{\Omega}}
%EndExpansion
\nabla u(x)\nabla v(x)dx=%
%TCIMACRO{\dint \nolimits_{\Omega}}%
%BeginExpansion
{\displaystyle\int\nolimits_{\Omega}}
%EndExpansion
g(x)v(x)dx
\]
for any $v\in H_{0}^{1}$. The function $g$ will be called the weak
Dirichlet-Laplacian and denoted by $(-\Delta)_{\omega}u$.

\begin{lemma}
The set $D(S)$ coincides with the set of all functions $u:\Omega
\rightarrow\mathbb{R}$ possessing the weak Dirichlet-Laplacian $(-\Delta
)_{\omega}u$ and%
\[
Su=(-\Delta)_{\omega}u
\]
for $u\in D(S).$
\end{lemma}

\begin{proof}
We have%
\begin{align*}
D(S)  &  =\{u\in H_{0}^{1};\ \left\langle u,\cdot\right\rangle _{V}\text{ is a
continuous functional on }H_{0}^{1}\text{ }\\
&  \text{with respect to the norm }\left\Vert \cdot\right\Vert _{L^{2}}\}\\
&  =\{u\in H_{0}^{1};\ \underset{g\in L^{2}}{\exists}\;\underset{v\in
H_{0}^{1}}{\forall}%
%TCIMACRO{\dint \nolimits_{\Omega}}%
%BeginExpansion
{\displaystyle\int\nolimits_{\Omega}}
%EndExpansion
\nabla u(x)\nabla v(x)dx=%
%TCIMACRO{\dint \nolimits_{\Omega}}%
%BeginExpansion
{\displaystyle\int\nolimits_{\Omega}}
%EndExpansion
g(x)v(x)dx\}\\
&  =\{u:\Omega\rightarrow\mathbb{R};\ u\text{ has the weak
Dirichlet-Laplacian}\}
\end{align*}
(for the second equality we extend functional $\left\langle u,\cdot
\right\rangle _{V}$ (preserving continuity) on $L^{2}$ to obtain $g\in L^{2}%
$). Of course, $Su=g=(-\Delta)_{\omega}u$.
\end{proof}

In the next, we shall denote the set $D(S)$ by $D((-\Delta)_{\omega})$ and $S$
by $(-\Delta)_{\omega}$. From Theorem \ref{Fried} it follows that
$(-\Delta)_{\omega}:D((-\Delta)_{\omega})\subset L^{2}\rightarrow L^{2}$ is
bijective, $T_{0}\subset(-\Delta)_{\omega}$ and%
\[
\lbrack(-\Delta)_{\omega}]^{\ast}=(-\Delta)_{\omega}\text{.}%
\]

Let us also observe that $-\Delta\subset(-\Delta)_{\omega}$ where
$-\Delta:H_{0}^{1}\cap H^{2}\subset L^{2}\rightarrow L^{2}$ is the (strong)
Dirichlet-Laplacian, i.e.%
\[
H_{0}^{1}\cap H^{2}\subset D((-\Delta)_{\omega})\text{ }%
\]
and
\[
(-\Delta)_{\omega}u=(-\Delta)u
\]
for $u\in H_{0}^{1}\cap H^{2}$. Indeed, if $u\in H_{0}^{1}\cap H^{2}$, then,
for any $i=1,...,N$, there exists a constant $C_{i}>0$ such that%
\[
\left\vert
%TCIMACRO{\dint \nolimits_{\Omega}}%
%BeginExpansion
{\displaystyle\int\nolimits_{\Omega}}
%EndExpansion
\frac{\partial u}{\partial x_{i}}(x)\frac{\partial\varphi}{\partial x_{i}%
}(x)dx\right\vert \leq C_{i}\left\Vert \varphi\right\Vert _{L^{2}}%
\]
for any $\varphi\in C_{c}^{\infty}$. Let us fix $v\in H_{0}^{1}$ and a
sequence $(\varphi_{m})\subset C_{c}^{\infty}$ such that $\varphi
_{m}\rightarrow v$ in $H^{1}$, i.e. $\varphi_{m}\rightarrow v$ in $L^{2}$ and
$\nabla\varphi_{m}\rightarrow\nabla v$ in $L^{2}$. Then,%
\[
\left\vert
%TCIMACRO{\dint \nolimits_{\Omega}}%
%BeginExpansion
{\displaystyle\int\nolimits_{\Omega}}
%EndExpansion
\frac{\partial u}{\partial x_{i}}(x)\frac{\partial v}{\partial x_{i}%
}(x)dx\right\vert \leq C_{i}\left\Vert v\right\Vert _{L^{2}}%
\]
and, consequently,%
\[
\left\vert
%TCIMACRO{\dint \nolimits_{\Omega}}%
%BeginExpansion
{\displaystyle\int\nolimits_{\Omega}}
%EndExpansion
\nabla u(x)\nabla v(x)dx\right\vert \leq%
%TCIMACRO{\dsum \nolimits_{i=1}^{n}}%
%BeginExpansion
{\displaystyle\sum\nolimits_{i=1}^{n}}
%EndExpansion
C_{i}\left\Vert v\right\Vert _{L^{2}}%
\]
for any $v\in H_{0}^{1}$. Thus, $u\in\{u\in H_{0}^{1};\ \left\langle
u,\cdot\right\rangle _{V}$ is a continuous functional on $H_{0}^{1}$ with
respect to the norm $\left\Vert \cdot\right\Vert _{L^{2}}\}=D((-\Delta
)_{\omega})$.

For $u\in H_{0}^{1}\cap H^{2}$, denote by $\Lambda_{u}$ the functional
$\left\langle u,\cdot\right\rangle _{V}:V\rightarrow L^{2}$ and by
$\widetilde{\Lambda}_{u}$ - a continuous extension of $\Lambda_{u}$ on $L^{2}$
with respect to $\left\Vert \cdot\right\Vert _{L^{2}}$. Clearly,%
\[
\widetilde{\Lambda}_{u}v=%
%TCIMACRO{\dint \nolimits_{\Omega}}%
%BeginExpansion
{\displaystyle\int\nolimits_{\Omega}}
%EndExpansion
(-\Delta)_{\omega}u(x)v(x)dx
\]
for $v\in L^{2}$. We shall show that $(-\Delta)_{\omega}u=(-\Delta)u$. Indeed,
let us consider an operator%
\[
\widetilde{\Gamma}_{u}:L^{2}\ni v\longmapsto%
%TCIMACRO{\dint \nolimits_{\Omega}}%
%BeginExpansion
{\displaystyle\int\nolimits_{\Omega}}
%EndExpansion
(-\Delta)u(x)v(x)dx\in\mathbb{R}\text{.}%
\]
For any $\varphi\in C_{c}^{\infty}$, we have%
\begin{align*}
\widetilde{\Gamma}_{u}\varphi &  =-%
%TCIMACRO{\dint \nolimits_{\Omega}}%
%BeginExpansion
{\displaystyle\int\nolimits_{\Omega}}
%EndExpansion%
%TCIMACRO{\dsum \nolimits_{i=1}^{N}}%
%BeginExpansion
{\displaystyle\sum\nolimits_{i=1}^{N}}
%EndExpansion
\frac{\partial^{2}u}{\partial x_{i}^{2}}(x)\varphi(x)=-%
%TCIMACRO{\dsum \nolimits_{i=1}^{N}}%
%BeginExpansion
{\displaystyle\sum\nolimits_{i=1}^{N}}
%EndExpansion%
%TCIMACRO{\dint \nolimits_{\Omega}}%
%BeginExpansion
{\displaystyle\int\nolimits_{\Omega}}
%EndExpansion
\frac{\partial}{\partial x_{i}}(\frac{\partial u}{\partial x_{i}}%
)(x)\varphi(x)\\
&  =%
%TCIMACRO{\dsum \nolimits_{i=1}^{N}}%
%BeginExpansion
{\displaystyle\sum\nolimits_{i=1}^{N}}
%EndExpansion%
%TCIMACRO{\dint \nolimits_{\Omega}}%
%BeginExpansion
{\displaystyle\int\nolimits_{\Omega}}
%EndExpansion
\frac{\partial u}{\partial x_{i}}(x)\frac{\partial\varphi}{\partial x_{i}%
}(x)dx=%
%TCIMACRO{\dint \nolimits_{\Omega}}%
%BeginExpansion
{\displaystyle\int\nolimits_{\Omega}}
%EndExpansion
\nabla u(x)\nabla\varphi(x)dx=\widetilde{\Lambda}_{u}\varphi.
\end{align*}
So, from the continuity of $\widetilde{\Lambda}_{u}$, $\widetilde{\Gamma}_{u}$
it follows that $\widetilde{\Lambda}_{u}=\widetilde{\Gamma}_{u}$ on $L^{2}$
and, consequently,%
\[
(-\Delta)_{\omega}u=(-\Delta)u.
\]

Finally, we have

\begin{theorem}
The operator
\[
(-\Delta)_{\omega}:D((-\Delta)_{\omega})\subset L^{2}\rightarrow L^{2}%
\]
is the self-adjoint extension of the operator $T_{0}$ and
\[
T_{0}\subset(-\Delta)\subset(-\Delta)_{\omega}%
\]
(the first relation is a standard result).
\end{theorem}

\begin{remark}
If $N=1$, then $(-\Delta)=(-\Delta)_{\omega}$ because $(-\Delta)$ is
self-adjoint and, consequently, no proper self-adjoint extension of it exists.
In the next section, we shall show that in some cases of $N>1$ the equality
$(-\Delta)=(-\Delta)_{\omega}$ holds true.
\end{remark}

\begin{remark}
\label{But}The operator $(-\Delta)_{\omega}$ is considered in \cite[Section
8.2]{Buttazzo} and called ''Laplace-Dirichlet
operator'' (without ''weak'')
and denoted by $(-\Delta)$.
\end{remark}

\subsection{Regularity of weak Dirichlet-Laplacian}

In this subsection, basing ourselves on results from \cite{Grisv1},
\cite{Grisv2} we describe two cases when weak Dirichlet-Laplace operator and
Dirichlet-Laplace operator coincide.

\subsubsection{$\Omega$ of class $C^{1,1}$}

To prove that $(-\Delta)=(-\Delta)_{\omega}$ it is sufficient to show that
$D((-\Delta)_{\omega})\subset D(-\Delta)$. Let us start with the following two
classical results. First of them is a particular case of \cite[Proposition
5.6.1]{Buttazzo} and the second one follows from \cite[Theorem 2.2.2.3]%
{Grisv1}.

\begin{theorem}
\label{tr}If $\Omega\subset\mathbb{R}^{N}$ is an open bounded set of class
$C^{1}$, then%
\[
H_{0}^{1}=\{u\in H^{1};\ tru=0\}
\]
where $tru$ denotes the trace of $u$ on $\Omega$.
\end{theorem}

\begin{theorem}
\label{Gris}If $\Omega\subset\mathbb{R}^{N}$ is an open bounded set of class
$C^{1,1}$, then, for every $g\in L^{2}$, there exists a unique solution $u\in
H^{2}$ of the problem:%
\[%
%TCIMACRO{\dint \nolimits_{\Omega}}%
%BeginExpansion
{\displaystyle\int\nolimits_{\Omega}}
%EndExpansion
\nabla u(x)\nabla v(x)dx=%
%TCIMACRO{\dint \nolimits_{\Omega}}%
%BeginExpansion
{\displaystyle\int\nolimits_{\Omega}}
%EndExpansion
g(x)v(x)dx
\]
for any $v\in\{v\in H^{1};\ trv=0\}$, with boundary condition
\[
tru=0.
\]

\end{theorem}

From the bijectivity of $(-\Delta)_{\omega}:D((-\Delta)_{\omega})\subset
L^{2}\rightarrow L^{2}$ and from Theorem \ref{tr} it follows that the unique
solution $u$ from Theorem \ref{Gris} is the unique solution of the problem%
\[
(-\Delta)_{\omega}u=g\text{.}%
\]
This simply implies

\begin{theorem}
If $\Omega\subset\mathbb{R}^{N}$ is an open bounded set of class $C^{1,1}$,
then\newline$(-\Delta)_{\omega}=(-\Delta)$.
\end{theorem}

\begin{proof}
Let $u\in D((-\Delta)_{\omega})$. Putting $g=(-\Delta)_{\omega}u$ we see that
$u\in H^{2}.$
\end{proof}

In the case of $\Omega\subset\mathbb{R}^{N}$ being of class $C^{2}$ such a
result is the classical Agmon-Douglis-Nirenberg theorem.

\subsubsection{$\Omega$ - convex polygon in $\mathbb{R}^{2}$}

From \cite[Theorem 2.4.3 and the comments following the theorem]{Grisv2} the
next result follows (\footnote{From \cite[Theorem 1.4.7]{Grisv2} it follows
that if $\Omega\subset\mathbb{R}^{2}$ is an open bounded polygon, then
\[
H_{0}^{1}=\{u\in H^{1};\ tr_{j}u=0\text{, }j=1,...,m\}
\]
where $m$ is a number of sides of the polygon and $tr_{j}u$ is a trace of $u$
on $j$-th side.}).

\begin{theorem}
\label{poly}If $\Omega\subset\mathbb{R}^{2}$ is an open bounded convex
polygon, then, for every $g\in L^{2}$, there exists a unique solution $u\in
H_{0}^{1}$ of the problem%
\[%
%TCIMACRO{\dint \nolimits_{\Omega}}%
%BeginExpansion
{\displaystyle\int\nolimits_{\Omega}}
%EndExpansion
\nabla u(x)\nabla v(x)dx=%
%TCIMACRO{\dint \nolimits_{\Omega}}%
%BeginExpansion
{\displaystyle\int\nolimits_{\Omega}}
%EndExpansion
g(x)v(x)dx,\ v\in H_{0}^{1},
\]
and $u\in H^{2}$.
\end{theorem}

Similarly, as in the case of $\Omega$ being of class $C^{1,1}$, we obtain

\begin{theorem}
If $\Omega\subset\mathbb{R}^{2}$ is an open bounded convex polygon, then
$(-\Delta)_{\omega}=(-\Delta)$.
\end{theorem}

\begin{proof}
Let $u\in D((-\Delta)_{\omega})$ and put $g=(-\Delta)_{\omega}u$. From Theorem
\ref{poly} it follows that $u$ being the unique solution of problem
$(-\Delta)_{\omega}u=g$ belongs to $H^{2}$.
\end{proof}

\subsection{\label{s}Spectrum of $(-\Delta)_{\omega}$}

Usually, in the literature, the set of eigenvalues of $(-\Delta)_{\omega}$ is
described. In this section, we show that the spectrum of $(-\Delta)_{\omega}$
contains only the eigenvalues (equivalently, the continuous part of the
spectrum is empty).

Let $\Omega\subset\mathbb{R}^{N}$ be an open bounded set. In \cite[Proposition
8.2.1]{Buttazzo} (see Remark \ref{But}) it is proved that the operator
$[(-\Delta)_{\omega}]^{-1}$ is bounded, compact, self-adjoint and positive
definite in $L^{2}$. Consequently (see \cite[Theorem 8.3.1]{Buttazzo}), it has
a denumerable spectrum consisting of $0$ and of finite multiplicity proper
values $\mu_{j}>0$ such that $0\longleftarrow\mu_{j}<...<\mu_{2}<\mu_{1}$. It
is easy to see that a number $\mu$ is a proper value of $[(-\Delta)_{\omega
}]^{-1}$ if and only if $\lambda=\frac{1}{\mu}$ is a proper value of
$(-\Delta)_{\omega}$ and the corresponding eigenspaces coincide. Applying
Proposition \ref{odwrot} with $A=(-\Delta)_{\omega}$ we see that%
\[
\lbrack(-\Delta)_{\omega}]^{-1}=%
%TCIMACRO{\dint \nolimits_{-\infty}^{\infty}}%
%BeginExpansion
{\displaystyle\int\nolimits_{-\infty}^{\infty}}
%EndExpansion
\frac{1}{\lambda}E(d\lambda)
\]
where $E$ is the spectral measure given by $(-\Delta)_{\omega}$ (the function
$\mathbb{R}\ni\lambda\longmapsto\lambda\in\mathbb{R}$ is different from $0$ on
$\mathbb{R\diagdown\{}0\}$ and from the fact that $\sigma((-\Delta)_{\omega
})\subset(0,\infty)$ (\footnote{From the fact that for $u\in D((-\Delta
)_{\omega})$ we have $u\in H_{0}^{1}$ and%
\[%
%TCIMACRO{\dint \nolimits_{\Omega}}%
%BeginExpansion
{\displaystyle\int\nolimits_{\Omega}}
%EndExpansion
\nabla u(x)\nabla v(x)dx=%
%TCIMACRO{\dint \nolimits_{\Omega}}%
%BeginExpansion
{\displaystyle\int\nolimits_{\Omega}}
%EndExpansion
(-\Delta)_{\omega}u(x)v(x)dx
\]
\par
for any $v\in H_{0}^{1}$ it follows that%
\[
\left\langle (-\Delta)_{\omega}u,u\right\rangle _{L^{2}}=\left\Vert \nabla
u\right\Vert _{L^{2}}\geq const\left\Vert u\right\Vert _{L^{2}}%
\]
where $const$ is a positive constant (the last inequality is the Poincare
inequality). Consequently (cf. \cite[Theorem IV.1.6.]{Mlak}) $\sigma
((-\Delta)_{\omega})\subset(0,\infty)$.
\par
{}}) it follows that $E(\{0\})=0$). So, using (\ref{ow}) we conclude that the
spectrum $\sigma((-\Delta)_{\omega})$ consists only of eigenvalues
$\lambda_{j}=\frac{1}{\mu_{j}}$ of $(-\Delta)_{\omega}$.

In the next (similarly, as in \cite{Buttazzo}), we shall count the eigenvalues
of $(-\Delta)_{\omega}$ according to their multiplicity, i.e. each
$\lambda_{j}$ is repeated $k_{j}$ times where $k_{j}$ is the multiplicity of
$\lambda_{j}$. Consequently, we obtain a sequence $0<\lambda_{1}\leq
\lambda_{2}\leq...\leq\lambda_{j}\rightarrow\infty$ (still denoted by
$(\lambda_{j})$) and a system $\{e_{j}\}$ of eigenfunctions of the operator
$(-\Delta)_{\omega}$, corresponding to $\lambda_{j}$, which is a Hilbertian
basis in $L^{2}$ (see \cite[Theorem 8.3.2]{Buttazzo}) (\footnote{In
\cite{Brezis} it is proved that one can choose $e_{j}\in H_{0}^{1}\cap
C^{\infty}$.}). Thus, for any $u\in L^{2}$ there exist real numbers $a_{j}$,
$j\in\mathbb{N}$, such that%
\[
u(t)=%
%TCIMACRO{\dsum }%
%BeginExpansion
{\displaystyle\sum}
%EndExpansion
a_{j}e_{j}(t)
\]
in $L^{2}$ and $\left\Vert u\right\Vert _{L^{2}}^{2}=%
%TCIMACRO{\dsum }%
%BeginExpansion
{\displaystyle\sum}
%EndExpansion
\left\vert a_{j}\right\vert ^{2}$.

For the form of the spectrum $\sigma((-\Delta)_{\omega})$ and eigenvectors of
$(-\Delta)_{\omega}$ in the case of $\Omega=(0,\pi)^{N}$ (cube in
$\mathbb{R}^{N}$) we refer the reader to \cite[Proposition 8.5.3]{Buttazzo}.
Let us only say that if $N=2$, then the first eigenvalue equals to $2$.

\section{Hilbert space $D([(-\Delta)_{\omega}]^{\beta})$}

Let us fix a number $\beta>0$, an open bounded set $\Omega\subset
\mathbb{R}^{N}$ and consider the operator
\[
\lbrack(-\Delta)_{\omega}]^{\beta}:D([(-\Delta)_{\omega}]^{\beta})\subset
L^{2}\rightarrow L^{2}%
\]
given by%
\begin{align*}
([(-\Delta)_{\omega}]^{\beta}u)(t)  &  =((%
%TCIMACRO{\dint \nolimits_{\sigma((-\Delta)_{\omega})}}%
%BeginExpansion
{\displaystyle\int\nolimits_{\sigma((-\Delta)_{\omega})}}
%EndExpansion
\lambda^{\beta}E(d\lambda))u)(t)\\
&  =(\lim(%
%TCIMACRO{\dint \nolimits_{\sigma((-\Delta)_{\omega})}}%
%BeginExpansion
{\displaystyle\int\nolimits_{\sigma((-\Delta)_{\omega})}}
%EndExpansion
(\lambda^{\beta})_{n}E(d\lambda)u))(t)=\sum\lambda_{j}^{\beta}a_{j}e_{j}(t)
\end{align*}
where%
\begin{multline*}
D([(-\Delta)_{\omega}]^{\beta})\\
=\{u(t)\in L^{2};\
%TCIMACRO{\dint \nolimits_{\sigma((-\Delta)_{\omega})}}%
%BeginExpansion
{\displaystyle\int\nolimits_{\sigma((-\Delta)_{\omega})}}
%EndExpansion
\left\vert \lambda^{\beta}\right\vert ^{2}\left\Vert E(d\lambda)u\right\Vert
^{2}=\sum((\lambda_{j})^{\beta})^{2}a_{j}^{2}<\infty,\\
\text{where }a_{j}\text{-s are such that }u(t)=(%
%TCIMACRO{\dint \nolimits_{\sigma((-\Delta)_{\omega})}}%
%BeginExpansion
{\displaystyle\int\nolimits_{\sigma((-\Delta)_{\omega})}}
%EndExpansion
1E(d\lambda)u)(t)=\sum a_{j}e_{j}(t)\},
\end{multline*}
$E$ is the spectral measure given by $(-\Delta)_{\omega}$ and the convergence
of the series is meant in $L^{2}$. Of course, $[(-\Delta)_{\omega}]^{\beta}$
is self-adjoint, the spectrum $\sigma([(-\Delta)_{\omega}]^{\beta})$ consists
of proper values $\lambda_{j}^{\beta}$, $j\in\mathbb{N}$, and eigenspaces
corresponding to $\lambda_{j}^{\beta}$-s are the same as eigenspaces for
$(-\Delta)_{\omega}$, corresponding to $\lambda_{j}$-s (it follows from a
general result concerning the power of any self-adjoint operator).

It is clear that if $0<\beta_{1}<\beta_{2}$, then%
\begin{equation}
D([(-\Delta)_{\omega}]^{\beta_{2}})\subset D([(-\Delta)_{\omega}]^{\beta_{1}%
}). \label{beta1beta22}%
\end{equation}
Indeed, if $u(t)=%
%TCIMACRO{\dsum }%
%BeginExpansion
{\displaystyle\sum}
%EndExpansion
a_{j}e_{j}(t)\in D([(-\Delta)_{\omega}]^{\beta_{2}})$, then%
\begin{align*}%
%TCIMACRO{\dsum \nolimits_{j=1}^{\infty}}%
%BeginExpansion
{\displaystyle\sum\nolimits_{j=1}^{\infty}}
%EndExpansion
((\lambda_{j})^{\beta_{1}})^{2}a_{j}^{2}  &  =%
%TCIMACRO{\dsum \nolimits_{\lambda_{j}<1}}%
%BeginExpansion
{\displaystyle\sum\nolimits_{\lambda_{j}<1}}
%EndExpansion
((\lambda_{j})^{\beta_{1}})^{2}a_{j}^{2}+%
%TCIMACRO{\dsum \nolimits_{\lambda_{j}\geq1}}%
%BeginExpansion
{\displaystyle\sum\nolimits_{\lambda_{j}\geq1}}
%EndExpansion
((\lambda_{j})^{\beta_{1}})^{2}a_{j}^{2}\\
&  \leq%
%TCIMACRO{\dsum \nolimits_{\lambda_{j}<1}}%
%BeginExpansion
{\displaystyle\sum\nolimits_{\lambda_{j}<1}}
%EndExpansion
((\lambda_{j})^{\beta_{1}})^{2}a_{j}^{2}+%
%TCIMACRO{\dsum \nolimits_{\lambda_{j}\geq1}}%
%BeginExpansion
{\displaystyle\sum\nolimits_{\lambda_{j}\geq1}}
%EndExpansion
((\lambda_{j})^{\beta_{2}})^{2}a_{j}^{2}\\
&  \leq%
%TCIMACRO{\dsum \nolimits_{\lambda_{j}<1}}%
%BeginExpansion
{\displaystyle\sum\nolimits_{\lambda_{j}<1}}
%EndExpansion
((\lambda_{j})^{\beta_{1}})^{2}a_{j}^{2}+%
%TCIMACRO{\dsum \nolimits_{j=1}^{\infty}}%
%BeginExpansion
{\displaystyle\sum\nolimits_{j=1}^{\infty}}
%EndExpansion
((\lambda_{j})^{\beta_{2}})^{2}a_{j}^{2}<\infty
\end{align*}
(the set of $\lambda_{j}<1$ is empty or finite).

Moreover, since $C_{c}^{\infty}\subset H_{0}^{1}\cap H^{2}\subset
D((-\Delta)_{\omega})$ and%
\[
(-\Delta)_{\omega}u=(-\Delta)u
\]
for $u\in C_{c}^{\infty}$, therefore using (\ref{formulann}) and
(\ref{beta1beta22}) we obtain%
\[
C_{c}^{\infty}\subset D([(-\Delta)_{\omega}]^{\beta})
\]
for any $\beta>0$.

Define in $D([(-\Delta)_{\omega}]^{\beta})$ the scalar product%
\[
\left\langle u,v\right\rangle _{\beta}=\left\langle u,v\right\rangle _{L^{2}%
}+\left\langle [(-\Delta)_{\omega}]^{\beta}u,[(-\Delta)_{\omega}]^{\beta
}v\right\rangle _{L^{2}}.
\]
The corresponding norm is given by%
\[
\left\Vert u\right\Vert _{\beta}=(\left\Vert u\right\Vert _{L^{2}}%
^{2}+\left\Vert [(-\Delta)_{\omega}]^{\beta}u\right\Vert _{L^{2}}^{2}%
)^{\frac{1}{2}}.
\]
Since $[(-\Delta)_{\omega}]^{\beta}$ is closed (being self-adjoint operator),
therefore it is easy to see that $D([(-\Delta)_{\omega}]^{\beta})$ with the
scalar product $\left\langle \cdot,\cdot\right\rangle _{\beta}$ is Hilbert space.

Let us also observe that the scalar product%
\[
\left\langle u,v\right\rangle _{\thicksim\beta}=\left\langle [(-\Delta
)_{\omega}]^{\beta}u,[(-\Delta)_{\omega}]^{\beta}v\right\rangle _{L^{2}}%
\]
determines the equivalent norm%
\[
\left\Vert u\right\Vert _{\thicksim\beta}=\left\Vert [(-\Delta)_{\omega
}]^{\beta}u\right\Vert _{L^{2}}.
\]
Indeed, we have%
\begin{align}
\left\Vert u\right\Vert _{L^{2}}^{2}  &  =%
%TCIMACRO{\dsum \nolimits_{j=1}^{\infty}}%
%BeginExpansion
{\displaystyle\sum\nolimits_{j=1}^{\infty}}
%EndExpansion
a_{j}^{2}=%
%TCIMACRO{\dsum \nolimits_{\lambda_{j}<1}}%
%BeginExpansion
{\displaystyle\sum\nolimits_{\lambda_{j}<1}}
%EndExpansion
a_{j}^{2}+%
%TCIMACRO{\dsum \nolimits_{\lambda_{j}\geq1}}%
%BeginExpansion
{\displaystyle\sum\nolimits_{\lambda_{j}\geq1}}
%EndExpansion
a_{j}^{2}\nonumber\\
&  \leq%
%TCIMACRO{\dsum \nolimits_{\lambda_{j}<1}}%
%BeginExpansion
{\displaystyle\sum\nolimits_{\lambda_{j}<1}}
%EndExpansion
a_{j}^{2}+%
%TCIMACRO{\dsum \nolimits_{\lambda_{j}\geq1}}%
%BeginExpansion
{\displaystyle\sum\nolimits_{\lambda_{j}\geq1}}
%EndExpansion
((\lambda_{j})^{\beta})^{2}a_{j}^{2}\label{eldwa}\\
&  \leq M_{\beta}%
%TCIMACRO{\dsum \nolimits_{\lambda_{j}<1}}%
%BeginExpansion
{\displaystyle\sum\nolimits_{\lambda_{j}<1}}
%EndExpansion
((\lambda_{j})^{\beta})^{2}a_{j}^{2}+%
%TCIMACRO{\dsum \nolimits_{\lambda_{j}\geq1}}%
%BeginExpansion
{\displaystyle\sum\nolimits_{\lambda_{j}\geq1}}
%EndExpansion
((\lambda_{j})^{\beta})^{2}a_{j}^{2}\nonumber
\end{align}
where%
\[
M_{\beta}=\max\{\frac{1}{((\lambda_{j})^{\beta})^{2}};\ \lambda_{j}<1\}>1
\]
(if the set $\{\lambda_{j};\ \lambda_{j}<1\}$ is empty, we put $M_{\beta}=1$).
Now,
\begin{align}
&  M_{\beta}%
%TCIMACRO{\dsum \nolimits_{\lambda_{j}<1}}%
%BeginExpansion
{\displaystyle\sum\nolimits_{\lambda_{j}<1}}
%EndExpansion
((\lambda_{j})^{\beta})^{2}a_{j}^{2}+%
%TCIMACRO{\dsum \nolimits_{\lambda_{j}\geq1}}%
%BeginExpansion
{\displaystyle\sum\nolimits_{\lambda_{j}\geq1}}
%EndExpansion
((\lambda_{j})^{\beta})^{2}a_{j}^{2}\label{eldwaprim}\\
&  \leq M_{\beta}%
%TCIMACRO{\dsum \nolimits_{j=1}^{\infty}}%
%BeginExpansion
{\displaystyle\sum\nolimits_{j=1}^{\infty}}
%EndExpansion
((\lambda_{j})^{\beta})^{2}a_{j}^{2}=M_{\beta}\left\Vert [(-\Delta)_{\omega
}]^{\beta}u\right\Vert _{L^{2}}^{2}=M_{\beta}\left\Vert u\right\Vert
_{\thicksim\beta}^{2}.\nonumber
\end{align}
Thus,%
\[
\left\Vert u\right\Vert _{\thicksim\beta}\leq\left\Vert u\right\Vert _{\beta
}\leq\sqrt{M_{\beta}+1}\left\Vert u\right\Vert _{\thicksim\beta}%
\]
which means the equivalence of the above norms.

In the next, we shall consider the space $D([(-\Delta)_{\omega}]^{\beta})$
with the scalar product $\left\langle \cdot,\cdot\right\rangle _{\thicksim
\beta}$.

\section{Equivalence of equations}

Let $E$ be the spectral measure for a self-adjoint operator $A:D(A)\subset
H\rightarrow H$ with $\sigma(A)\subset\lbrack0,\infty)$, $\alpha_{i}%
\in\mathbb{R}$ for $i=0,...,k$ ($k\in\mathbb{N}\cup\{0\}$) and $0\leq\beta
_{0}<\beta_{1}<...<\beta_{k}$. Fact that the operator $w(A)$ where%
\begin{equation}
w:\mathbb{R}\ni\lambda\rightarrow\left\{
\begin{array}
[c]{ccc}%
0 & ; & \lambda<0\\
\alpha_{k}\lambda^{\beta_{k}}+...+\alpha_{1}\lambda^{\beta_{1}}+\alpha
_{0}\lambda^{\beta_{0}} & ; & \lambda\geq0
\end{array}
\right.  , \label{wielomian}%
\end{equation}
is self-adjoint means that its domain satisfies the equality%
\begin{gather}
D(w(A))=\{u\in L^{2};\ \text{\textit{there exists} }z\in L^{2}\text{
\textit{such that}}\label{sa11}\\%
%TCIMACRO{\dint \nolimits_{\Omega}}%
%BeginExpansion
{\displaystyle\int\nolimits_{\Omega}}
%EndExpansion
u(t)w(A)v(t)dt=%
%TCIMACRO{\dint \nolimits_{\Omega}}%
%BeginExpansion
{\displaystyle\int\nolimits_{\Omega}}
%EndExpansion
z(t)v(t)dt\text{ \textit{for any} }v\in D(w(A))\}\nonumber
\end{gather}
and%
\begin{equation}
w(A)u=z \label{sa2}%
\end{equation}
for $u\in D(w(A))$ (\footnote{If the numbers $\beta_{0}$, $\beta_{1}%
$,...,$\beta_{k}$ are positive integers (including zero), then one can omitte
the assumption $\sigma(A)\subset\lbrack0,\infty)$ and consider the function%
\[
w(\lambda)=\alpha_{k}\lambda^{\beta_{k}}+...+\alpha_{1}\lambda^{\beta_{1}%
}+\alpha_{0}\lambda^{\beta_{0}},\ \lambda\in\mathbb{R}\text{.}%
\]
})

From (\ref{formulan}) it follows that $u\in D(w^{2}(A))$ (\footnote{Clearly,
$w^{2}(A)=\sum_{i,j=0}^{k}\alpha_{i}\alpha_{j}A^{\beta_{i}+\beta_{j}}$.}) if
and only if $u\in D(w(A))$, $w(A)u\in D(w(A))$ and, in such a case,%
\begin{equation}
w(A)(w(A)u)=w^{2}(A)u\text{.} \label{betabeta2beta}%
\end{equation}
Using this fact and (\ref{sa11}), (\ref{sa2}), we obtain

\begin{theorem}
\label{slabe}If $g\in L^{2}$, then $u\in D(w^{2}(A))$ and%
\begin{equation}
w^{2}(A)u=g \label{strong}%
\end{equation}
if and only if $u\in D(w(A))$ and%
\begin{equation}%
%TCIMACRO{\dint \nolimits_{\Omega}}%
%BeginExpansion
{\displaystyle\int\nolimits_{\Omega}}
%EndExpansion
w(A)u(t)w(A)v(t)dt=%
%TCIMACRO{\dint \nolimits_{\Omega}}%
%BeginExpansion
{\displaystyle\int\nolimits_{\Omega}}
%EndExpansion
g(t)v(t)dt \label{weak}%
\end{equation}
for any $v\in D(w(A))$.
\end{theorem}

\begin{remark}
\label{newvarmethod}The above theorem states that $u$ is the strong solution
to problem (\ref{strong}) if and only if it is the weak one (in a sense).
Consequently, it can be obtained with the aid of a variational method (see
Section \ref{8}). Let us observe that in the case of $A=(-\Delta)_{\omega}$,
$w(\lambda)=\lambda^{\frac{1}{2}}$ and $\Omega\subset\mathbb{R}^{N}$ being an
open bounded set, the unique strong solution of problem
\[
(-\Delta)_{\omega}u=g
\]
(in fact, weak solution to the equation $(-\Delta)u=g$) is a function $u\in
H_{0}^{1}$ such that%
\[%
%TCIMACRO{\dint \nolimits_{\Omega}}%
%BeginExpansion
{\displaystyle\int\nolimits_{\Omega}}
%EndExpansion
\nabla u(t)\nabla v(t)dt=%
%TCIMACRO{\dint \nolimits_{\Omega}}%
%BeginExpansion
{\displaystyle\int\nolimits_{\Omega}}
%EndExpansion
g(t)v(t)dt
\]
for any $v\in H_{0}^{1}$. From the above theorem it follows that $u\in
D([(-\Delta)_{\omega}]^{\frac{1}{2}})$ and condition (\ref{weak}) holds true,
i.e.%
\[%
%TCIMACRO{\dint \nolimits_{\Omega}}%
%BeginExpansion
{\displaystyle\int\nolimits_{\Omega}}
%EndExpansion
[(-\Delta)_{\omega}]^{\frac{1}{2}}u(t)[(-\Delta)_{\omega}]^{\frac{1}{2}%
}v(t)dt=%
%TCIMACRO{\dint \nolimits_{\Omega}}%
%BeginExpansion
{\displaystyle\int\nolimits_{\Omega}}
%EndExpansion
g(t)v(t)dt
\]
If additionally $\Omega\subset\mathbb{R}^{N}$ is of class $C^{1,1}$ or convex
polygon in $\mathbb{R}^{2}$, then the unique strong solution to (\ref{strong})
is a function $u\in H_{0}^{1}\cap H^{2}$ such that $(-\Delta)u=g$, i.e. the
strong solution to the equation $(-\Delta)u=g$.

Finally, let us point out that even in the case of $N=1$ and $\Omega=(0,\pi)$
the operator $[(-\Delta)_{\omega}]^{\frac{1}{2}}=(-\Delta)^{\frac{1}{2}}$
differs from the operator
\[
H_{0}^{1}\subset L^{2}\rightarrow L^{2},
\]%
\[
x\longmapsto\nabla x=x^{\prime}%
\]
(the last one is not self-adjoint). So, we have a new variational method for
study the equation $(-\Delta)u=g.$
\end{remark}

\section{Compactness of the inverse $(w^{2}((-\Delta)_{\omega}))^{-1}$}

Let us consider the operator $w((-\Delta)_{\omega})$ assuming additionally
that $\alpha_{i}>0$ for $i=0,...,k$. From (\ref{beta1beta22}) it follows that
$D(w((-\Delta)_{\omega}))=D([(-\Delta)_{\omega}]^{\beta_{k}})$. Introduce in
$D([(-\Delta)_{\omega}]^{\beta_{k}})$ a new scalar product%
\[
\left\langle u,v\right\rangle _{w((-\Delta)_{\omega})}=\left\langle
w((-\Delta)_{\omega})u,w((-\Delta)_{\omega})v\right\rangle _{L^{2}}\text{.}%
\]
We have

\begin{lemma}
\label{rownwiel}The scalar products $\left\langle \cdot,\cdot\right\rangle
_{\thicksim\beta_{k}}$ and $\left\langle \cdot,\cdot\right\rangle
_{w((-\Delta)_{\omega})}$ generate the equivalent norms%
\[
\left\Vert u\right\Vert _{\thicksim\beta_{k}}=\left\Vert [(-\Delta)_{\omega
}]^{\beta_{k}}u\right\Vert _{L^{2}}%
\]
and%
\[
\left\Vert u\right\Vert _{w((-\Delta)_{\omega})}=\left\Vert w((-\Delta
)_{\omega})u\right\Vert _{L^{2}}%
\]
in $D([(-\Delta)_{\omega}]^{\beta_{k}})$.
\end{lemma}

\begin{proof}
First, let us observe that if $\beta_{i}<\beta_{j}$, then (see
(\ref{beta1beta2}))%
\begin{align*}
&  \alpha_{i}\alpha_{j}\left\langle [(-\Delta)_{\omega}]^{\beta_{i}%
}u,[(-\Delta)_{\omega}]^{\beta_{j}}u\right\rangle _{L^{2}}\\
&  =\alpha_{i}\alpha_{j}\left\langle [(-\Delta)_{\omega}]^{\beta_{i}%
}u,[(-\Delta)_{\omega}]^{\beta_{j}-\beta_{i}}([(-\Delta)_{\omega}]^{\beta_{i}%
}u)\right\rangle _{L^{2}}\\
&  =\alpha_{i}\alpha_{j}\left\langle [(-\Delta)_{\omega}]^{\frac{\beta
_{j}-\beta_{i}}{2}}([(-\Delta)_{\omega}]^{\beta_{i}}u),[(-\Delta)_{\omega
}]^{\frac{\beta_{j}-\beta_{i}}{2}}([(-\Delta)_{\omega}]^{\beta_{i}%
}u)\right\rangle _{L^{2}}\\
&  =\alpha_{i}\alpha_{j}\left\Vert [(-\Delta)_{\omega}]^{\frac{\beta_{j}%
-\beta_{i}}{2}+\beta_{i}}u\right\Vert _{L^{2}}^{2}\geq0.
\end{align*}
Using this property we obtain%
\begin{align*}
\left\Vert u\right\Vert _{\thicksim\beta_{k}}^{2}  &  =\frac{1}{\alpha_{k}%
^{2}}\left\Vert \alpha_{k}[(-\Delta)_{\omega}]^{\beta_{k}}u\right\Vert
_{L^{2}}^{2}\leq\frac{1}{\alpha_{k}^{2}}\left\Vert w((-\Delta)_{\omega
})u\right\Vert _{L^{2}}^{2}\\
&  \leq\frac{C_{1}}{\alpha_{k}^{2}}%
%TCIMACRO{\dsum \nolimits_{i=0}^{k}}%
%BeginExpansion
{\displaystyle\sum\nolimits_{i=0}^{k}}
%EndExpansion
\left\Vert [(-\Delta)_{\omega}]^{\beta_{i}}u\right\Vert _{L^{2}}^{2}%
=\frac{C_{1}}{\alpha_{k}^{2}}%
%TCIMACRO{\dsum \nolimits_{i=0}^{k}}%
%BeginExpansion
{\displaystyle\sum\nolimits_{i=0}^{k}}
%EndExpansion
(%
%TCIMACRO{\dsum \nolimits_{j=1}^{\infty}}%
%BeginExpansion
{\displaystyle\sum\nolimits_{j=1}^{\infty}}
%EndExpansion
((\lambda_{j})^{\beta_{i}})^{2}a_{j}^{2})\\
&  \leq\frac{C_{1}}{\alpha_{k}^{2}}(%
%TCIMACRO{\dsum \nolimits_{j=1}^{\infty}}%
%BeginExpansion
{\displaystyle\sum\nolimits_{j=1}^{\infty}}
%EndExpansion
((\lambda_{j})^{\beta_{k}})^{2}a_{j}^{2}+kC_{2}^{2}%
%TCIMACRO{\dsum \nolimits_{j=1}^{\infty}}%
%BeginExpansion
{\displaystyle\sum\nolimits_{j=1}^{\infty}}
%EndExpansion
((\lambda_{j})^{\beta_{k}})^{2}a_{j}^{2})\\
&  =\frac{C_{1}}{\alpha_{k}^{2}}(1+kC_{2}^{2})%
%TCIMACRO{\dsum \nolimits_{j=1}^{\infty}}%
%BeginExpansion
{\displaystyle\sum\nolimits_{j=1}^{\infty}}
%EndExpansion
((\lambda_{j})^{\beta_{k}})^{2}a_{j}^{2}=\frac{C_{1}}{\alpha_{k}^{2}}%
(1+kC_{2}^{2})\left\Vert u\right\Vert _{\thicksim\beta_{k}}^{2}%
\end{align*}
where $u(x)=%
%TCIMACRO{\dsum \nolimits_{j=1}^{\infty}}%
%BeginExpansion
{\displaystyle\sum\nolimits_{j=1}^{\infty}}
%EndExpansion
a_{j}e_{j}(x)$ ($\{e_{j};\ j\in\mathbb{N}$\} is a hilbertian basis in $L^{2}$
consisting of eigenfunctions corresponding to eigenvalues $\lambda_{j}$),
$C_{1}>0$ is a constant that does not depend on $u$ and $C_{2}$ is such that%
\[
(\lambda_{j})^{\beta_{i}}\leq C_{2}(\lambda_{j})^{\beta_{k}}%
\]
for any $i=0,...,k-1$ and $j\in\{j\in\mathbb{N}$, $\lambda_{j}<1\}$ (if the
set $\{j\in\mathbb{N}$, $\lambda_{j}<1\}$ is empty, we put $C_{2}=1$).
\end{proof}

From the above lemma the following proposition follows.

\begin{proposition}
The space $D(w((-\Delta)_{\omega}))=D([(-\Delta)_{\omega}]^{\beta_{k}})$ with
the scalar product $\left\langle \cdot,\cdot\right\rangle _{w((-\Delta
)_{\omega})}$ is complete.
\end{proposition}

Now, let us fix $g\in L^{2}$ and consider the equation%
\[
w^{2}((-\Delta)_{\omega})u=g
\]
in $D(w^{2}((-\Delta)_{\omega}))$. According to Theorem \ref{slabe}, to show
that there exists a unique solution to this equation it is sufficient to prove
that there exists a unique function $u\in D(w((-\Delta)_{\omega}))$ such that%
\[%
%TCIMACRO{\dint \nolimits_{\Omega}}%
%BeginExpansion
{\displaystyle\int\nolimits_{\Omega}}
%EndExpansion
w(A)u(t)w(A)v(t)dt=%
%TCIMACRO{\dint \nolimits_{\Omega}}%
%BeginExpansion
{\displaystyle\int\nolimits_{\Omega}}
%EndExpansion
g(t)v(t)dt
\]
for any $v\in D(w((-\Delta)_{\omega}))$. Indeed, the functional%
\[
D(w((-\Delta)_{\omega}))\ni u\longmapsto%
%TCIMACRO{\dint \nolimits_{\Omega}}%
%BeginExpansion
{\displaystyle\int\nolimits_{\Omega}}
%EndExpansion
g(x)u(x)dx\in\mathbb{R}%
\]
is linear and continuous with respect to the norm $\left\Vert u\right\Vert
_{w((-\Delta)_{\omega})}$ (continuity follows from (\ref{eldwa}),
(\ref{eldwaprim}) and Lemma \ref{rownwiel}). So, from the Riesz theorem it
follows that there exists a unique function $u_{g}\in D(w((-\Delta)_{\omega
}))$ such that%
\[
\left\langle u_{g},v\right\rangle _{w((-\Delta)_{\omega})}=%
%TCIMACRO{\dint \nolimits_{\Omega}}%
%BeginExpansion
{\displaystyle\int\nolimits_{\Omega}}
%EndExpansion
g(x)v(x)dx
\]
for any $v\in D(w((-\Delta)_{\omega}))$, i.e.%
\[%
%TCIMACRO{\dint \nolimits_{\Omega}}%
%BeginExpansion
{\displaystyle\int\nolimits_{\Omega}}
%EndExpansion
w(A)u_{g}(t)w(A)v(t)dt=%
%TCIMACRO{\dint \nolimits_{\Omega}}%
%BeginExpansion
{\displaystyle\int\nolimits_{\Omega}}
%EndExpansion
g(t)v(t)dt
\]
for any $v\in D(w((-\Delta)_{\omega}))$. Thus, we have proved

\begin{theorem}
For any function $g\in L^{2}$, there exists a unique solution $u_{g}\in
D(w^{2}((-\Delta)_{\omega}))$ to the equation%
\[
w^{2}((-\Delta)_{\omega})u=g.
\]

\end{theorem}

So, the operator $w^{2}((-\Delta)_{\omega}):D(w^{2}((-\Delta)_{\omega
}))\subset L^{2}\rightarrow L^{2}$ is bijective and, consequently, there
exists an inverse operator%
\[
(w^{2}((-\Delta)_{\omega}))^{-1}:L^{2}\rightarrow L^{2}%
\]
defined on the whole space $L^{2}$. Moreover, for every $g\in L^{2}$, we have%
\begin{align*}
\left\Vert (w^{2}((-\Delta)_{\omega}))^{-1}g\right\Vert _{L^{2}}^{2}  &
=\left\Vert (w^{2}((-\Delta)_{\omega}))^{-1}(w^{2}((-\Delta)_{\omega}%
)u_{g})\right\Vert _{L^{2}}^{2}=\left\Vert u_{g}\right\Vert _{L^{2}}^{2}\\
&  \leq M_{2\beta_{k}}\left\Vert u_{g}\right\Vert _{\thicksim2\beta_{k}}%
^{2}\leq\frac{M_{2\beta_{k}}}{\alpha_{k}^{4}}\left\Vert u_{g}\right\Vert
_{w^{2}((-\Delta)_{\omega})}^{2}\\
&  =\frac{M_{2\beta_{k}}}{\alpha_{k}^{4}}\left\Vert w^{2}((-\Delta)_{\omega
})u_{g}\right\Vert _{L^{2}}^{2}=\frac{M_{2\beta_{k}}}{\alpha_{k}^{4}%
}\left\Vert g\right\Vert _{L^{2}}^{2},
\end{align*}
i.e. $(w^{2}((-\Delta)_{\omega}))^{-1}$ is bounded. Using Proposition
\ref{odwrot} with the operator $A=(-\Delta)_{\omega}$ and the function
$b(\lambda)=w^{2}(\lambda)$, we assert that%
\[
(w^{2}((-\Delta)_{\omega}))^{-1}=%
%TCIMACRO{\dint \nolimits_{-\infty}^{\infty}}%
%BeginExpansion
{\displaystyle\int\nolimits_{-\infty}^{\infty}}
%EndExpansion
\frac{1}{w^{2}(\lambda)}E(d\lambda).
\]
Thus, the operator $(w^{2}((-\Delta)_{\omega}))^{-1}$ is self-adjoint and (see
(\ref{ow})) the spectrum $\sigma(([(-\Delta)_{\omega}]^{\beta})^{-1})$
consists of $0$ and proper values $\mu_{j}=\frac{1}{w^{2}(\lambda_{j})}$
($\lambda_{j}$, $j\in\mathbb{N}$, are proper values of $(-\Delta)_{\omega}$)
such that $0\longleftarrow\mu_{j}<...<\mu_{2}<\mu_{1}$ (we used here the fact
that $\alpha_{j}>0$ for $j=0,...,k$ and, consequently, $w^{2}(\lambda)$ is
increasing, $w^{2}(\lambda)\rightarrow\infty$ as $\lambda\rightarrow\infty$
and $w^{2}(\lambda)\neq0$ for $\lambda>0$). Multiplicities of $\mu_{j}$ are
the same as multiplicities of $w^{2}(\lambda_{j})$ and, in fact, the same as
multiplicities of $\lambda_{j}$.

Finally, we have the operator $(w^{2}((-\Delta)_{\omega}))^{-1}$ which is
defined on $L^{2}$, bounded, self-adjoint with countable spectrum consisting
of $0$ and of finite multiplicity eigenvalues tending to $0$. Using \cite[Part
VI.6]{Mlak} we obtain

\begin{theorem}
\label{comp}The operator $(w^{2}((-\Delta)_{\omega}))^{-1}$ is compact, i.e.
the image of any bounded set in $L^{2}$ is relatively compact in $L^{2}$.
\end{theorem}

\begin{remark}
The case $w(\lambda)=\lambda^{\frac{1}{2}}$ of the above theorem is proved in
\cite[Proposition 8.2.1]{Buttazzo} and that proof is based on the
Rellich-Kondrakov theorem.
\end{remark}

Using the above theorem we obtain the following property.

\begin{proposition}
\label{weak conver 2}If $u_{k}\rightharpoonup u_{0}$ weakly in $D(w((-\Delta
)_{\omega}))$, then $u_{k}\rightarrow u_{0}$ strongly in $L^{2}$ and
$w((-\Delta)_{\omega})u_{k}\rightharpoonup w((-\Delta)_{\omega})u_{0}$ weakly
in $L^{2}$.
\end{proposition}

\begin{proof}
First, let us assume that $w(\lambda)=z^{2}(\lambda)$ where $z$ is a
polynomial of type (\ref{wielomian}) with positive coefficients $\alpha_{i}$.
From the continuity of the linear operators%
\[
D(z^{2}((-\Delta)_{\omega}))\ni u\longmapsto u\in L^{2},
\]%
\[
D(z^{2}((-\Delta)_{\omega}))\ni u\longmapsto z^{2}((-\Delta)_{\omega})u\in
L^{2}%
\]
it follows that $u_{k}\rightharpoonup u_{0}$ weakly in $L^{2}$ and
$z^{2}((-\Delta)_{\omega})u_{k}\rightharpoonup z^{2}((-\Delta)_{\omega})u_{0}$
weakly in $L^{2}$. Theorem \ref{comp} implies that the sequence $(u_{k})$
contains a subsequence $(u_{k_{i}})$ converging strongly in $L^{2}$ to a
limit. Of course, this limit is the function $u_{0}$, i.e. $u_{k_{i}%
}\rightarrow u_{0}$ strongly in $L^{2}$. Supposing contrary and repeating the
above argumentation we assert that $u_{k}\rightarrow u_{0}$ strongly in
$L^{2}$.

Now, let us consider any polynomial $w(\lambda)$ of type (\ref{wielomian}).
Clearly, weak convergence $u_{k}\rightharpoonup u_{0}$ in $D(w((-\Delta
)_{\omega}))$ implies the weak convergence $w((-\Delta)_{\omega}%
)u_{k}\rightharpoonup w((-\Delta)_{\omega})u_{0}$ in $L^{2}$. Moreover,
\[
D(w((-\Delta)_{\omega}))=D([(-\Delta)_{\omega}]^{\beta_{k}})=D([(-\Delta
)_{\omega}]^{2\frac{\beta_{k}}{2}})=D(z^{2}((-\Delta)_{\omega}))
\]
where $z(\lambda)=\lambda^{\frac{\beta_{k}}{2}}$. Applying the proved case of
the proposition to the polynomial $z(\lambda)$ we assert that $u_{k}%
\rightarrow u_{0}$ strongly in $L^{2}$.
\end{proof}

\section{$\label{8}$A boundary value problem}

Let us consider boundary value problem (\ref{problemgeneral}). By a solution
we mean a function $u\in D(w^{2}((-\Delta)_{\omega}))$ satisfying
(\ref{problemgeneral}) a.e. on $\Omega$. According to Theorem \ref{slabe}, to
derive existence of a solution to (\ref{problemgeneral}) it is sufficient to
show that there exists $u\in D(w((-\Delta)_{\omega}))$ such that%
\begin{equation}%
%TCIMACRO{\dint \nolimits_{\Omega}}%
%BeginExpansion
{\displaystyle\int\nolimits_{\Omega}}
%EndExpansion
w((-\Delta)_{\omega})u(t)w((-\Delta)_{\omega})v(t)dt=%
%TCIMACRO{\dint \nolimits_{\Omega}}%
%BeginExpansion
{\displaystyle\int\nolimits_{\Omega}}
%EndExpansion
D_{u}F(x,u(x))v(t)dt \label{var}%
\end{equation}
for any $v\in D(w((-\Delta)_{\omega}))$. In such a case, solution to
(\ref{var}) is a solution to (\ref{problemgeneral}). Of course, such a point
$u$ is a critical point of the functional%
\begin{equation}
f:D(w((-\Delta)_{\omega}))\ni u\longmapsto%
%TCIMACRO{\dint \nolimits_{\Omega}}%
%BeginExpansion
{\displaystyle\int\nolimits_{\Omega}}
%EndExpansion
(\frac{1}{2}\left\vert w((-\Delta)_{\omega})u(x)\right\vert ^{2}%
-F(x,u(x)))dx\in\mathbb{R} \label{fdzial}%
\end{equation}
(clearly, under assumptions guaranteeing Gateaux differentiability of $f$).

\subsection{Gateaux differentiability of $F$}

Assume that function $F$ is measurable in $x\in\Omega$, continuously
differentiable in $u\in\mathbb{R}$ and%
\begin{equation}
\left\vert F(x,u)\right\vert \leq a\left\vert u\right\vert ^{2}+b(x)
\label{fwzrost0}%
\end{equation}%
\begin{equation}
\left\vert D_{u}F(x,u)\right\vert \leq c\left\vert u\right\vert +d(x)
\label{fwzrost1}%
\end{equation}
for $x\in\Omega$ a.e., $u\in\mathbb{R}$, where $a$, $c\geq0$ and $b\in L^{1}$,
$d\in L^{2}$. We have

\begin{proposition}
\label{Diff}Functional $f$ is differentiable in Gateaux sense and the
differential $f^{\prime}(u):D(w((-\Delta)_{\omega}))\rightarrow\mathbb{R}$ of
$f$ at $u$ is given by%
\[
f^{\prime}(u)v=%
%TCIMACRO{\dint \nolimits_{\Omega}}%
%BeginExpansion
{\displaystyle\int\nolimits_{\Omega}}
%EndExpansion
w((-\Delta)_{\omega})u(t)w((-\Delta)_{\omega})v(t)-D_{u}F(x,u(x))v(t)dt
\]
for $v\in D(w((-\Delta)_{\omega}))$.
\end{proposition}

\begin{proof}
Of course, the first term of $f$, equal to $\frac{1}{2}\left\Vert u\right\Vert
_{w((-\Delta)_{\omega})}^{2}$, is Gateaux (even continuously Gateaux)
differentiable and its Gateaux differential at $u$ is of the form%
\[
D(w((-\Delta)_{\omega}))\ni v\longmapsto\left\langle u,v\right\rangle
_{w((-\Delta)_{\omega})}.
\]
So, let us consider the mapping%
\[
g:D(w((-\Delta)_{\omega}))\ni u\longmapsto%
%TCIMACRO{\dint \nolimits_{\Omega}}%
%BeginExpansion
{\displaystyle\int\nolimits_{\Omega}}
%EndExpansion
F(x,u(x))dx\in\mathbb{R}.
\]
We shall show that the mapping%
\[
g^{\prime}(u):D(w((-\Delta)_{\omega}))\ni v\longmapsto%
%TCIMACRO{\dint \nolimits_{\Omega}}%
%BeginExpansion
{\displaystyle\int\nolimits_{\Omega}}
%EndExpansion
D_{u}F(x,u(x))v(x)dx\in\mathbb{R},
\]
is Gateaux differential of $g$ at $u$.

Linearity and continuity of the mapping $g^{\prime}(u)$ are obvious. To prove
that $g^{\prime}(u)$ is the Gateaux differential of $g$ at $u$ we shall show
that, for any $v\in D(w((-\Delta)_{\omega}))$,%
\begin{multline*}
\left\vert \frac{g(u+\tau_{k}v)-g(u)}{\tau_{k}}-g^{\prime}(u)v\right\vert \\
=%
%TCIMACRO{\dint \nolimits_{\Omega}}%
%BeginExpansion
{\displaystyle\int\nolimits_{\Omega}}
%EndExpansion
\left\vert \frac{F(x,u(x)+\tau_{k}v(x))-F(x,u(x))}{\tau_{k}}-D_{u}%
F(x,u(x))v(x)\right\vert dt\\
\rightarrow0
\end{multline*}
for any sequence $(\tau_{k})\subset(-1,1)$ such that $\tau_{k}\rightarrow0$.
Indeed, the sequence of functions%
\[
x\longmapsto\frac{F(x,u(x)+\tau_{k}v(x))-F(x,u(x))}{\tau_{k}}-D_{u}%
F(x,u(x))v(x)
\]
converges pointwise a.e. on $\Omega$ to the zero function (by
differentiability of $F$). Moreover, using the mean value theorem we say that
this sequence is bounded by a function from $L^{1}$:%
\begin{multline*}
\left\vert \frac{F(x,u(x)+\tau_{k}v(x))-F(x,u(x))}{\tau_{k}}-D_{u}%
F(x,u(x))v(x)\right\vert \\
=\left\vert D_{u}F(x,u(x)+s_{x,k\tau k}v(x))v(x)-D_{u}F(x,u(x))v(x)\right\vert
\\
\leq(c(2\left\vert u(x)\right\vert +\left\vert v(x)\right\vert
)+2d(x))\left\vert v(x)\right\vert
\end{multline*}
where $s_{x,k}\in(0,1)$. Thus, using the Lebesgue dominated convergence
theorem we state that $g^{\prime}(u)$ is Gateaux differential of $g$ at $u$.
\end{proof}

\subsection{Existence of a solution}

First, we shall prove the following two propositions.

\begin{proposition}
If there exist constants $A<\frac{\alpha_{k}^{2}}{M_{\beta_{k}}}$, $B$,
$C\in\mathbb{R}$ such that%
\begin{equation}
F(x,u)\leq\frac{A}{2}\left\vert u\right\vert ^{2}+B\left\vert u\right\vert
+C\text{ } \label{fwzrost2}%
\end{equation}
for $x\in\Omega$ a.e., $u\in\mathbb{R}$, then the functional (\ref{fdzial}) is
coercive, i.e. $f(u)\rightarrow\infty$ as $\left\Vert u\right\Vert
_{w((-\Delta)_{\omega})}\rightarrow\infty$.
\end{proposition}

\begin{proof}
Let us assume, without loss of the generality, that $A,B\geq0$. For any $u\in
D(w((-\Delta)_{\omega}))$, we have
\begin{align*}
f(u)  &  =%
%TCIMACRO{\dint \nolimits_{\Omega}}%
%BeginExpansion
{\displaystyle\int\nolimits_{\Omega}}
%EndExpansion
(\frac{1}{2}\left\vert w((-\Delta)_{\omega})u(x)\right\vert ^{2}%
-F(x,u(x)))dx\\
&  \geq\frac{1}{2}\left\Vert u\right\Vert _{w((-\Delta)_{\omega})}^{2}%
-\frac{A}{2}\left\Vert u\right\Vert _{L^{2}}^{2}-B\sqrt{\left\vert
\Omega\right\vert }\left\Vert u\right\Vert _{L^{2}}-C\left\vert \Omega
\right\vert \\
&  \geq\frac{1}{2}\left\Vert u\right\Vert _{w((-\Delta)_{\omega})}^{2}%
-\frac{A}{2}M_{\beta_{k}}\left\Vert u\right\Vert _{\thicksim\beta_{k}}%
^{2}-B\sqrt{\left\vert \Omega\right\vert M_{\beta_{k}}}\left\Vert u\right\Vert
_{\thicksim\beta_{k}}-C\left\vert \Omega\right\vert \\
&  \geq\frac{1}{2}(1-\frac{AM_{\beta_{k}}}{\alpha_{k}^{2}})\left\Vert
u\right\Vert _{w((-\Delta)_{\omega})}^{2}-B\sqrt{\left\vert \Omega\right\vert
M_{\beta_{k}}}\frac{1}{\alpha_{k}}\left\Vert u\right\Vert _{w((-\Delta
)_{\omega})}-C\left\vert \Omega\right\vert
\end{align*}
where $\left\vert \Omega\right\vert $ is the Lebesgue measure of $\Omega$. It
means that $f$ is coercive.
\end{proof}

\begin{proposition}
Functional (\ref{fdzial}) is weakly sequentially lower semicontinuous.
\end{proposition}

\begin{proof}
Weak sequential lower semicontinuity of a norm in Banach space is a classical
result. So, it sufficient to show that the functional
\[
D(w((-\Delta)_{\omega}))\ni u\longmapsto%
%TCIMACRO{\dint \nolimits_{\Omega}}%
%BeginExpansion
{\displaystyle\int\nolimits_{\Omega}}
%EndExpansion
F(x,u(x)))dx\in\mathbb{R}%
\]
is weakly continuous. But this fact follows immediately from Lebesgue
dominated convergence theorem. Indeed, the weak convergence of a sequence
$(u_{n})$ to $u_{0}$ in $D(w((-\Delta)_{\omega}))$ implies (see Proposition
\ref{weak conver 2}) the convergence $u_{n}\rightarrow u_{0}$ in $L^{2}$. From
\cite[Theorem 4.9]{Brezis} it follows that one can choose a subsequence
$(u_{n_{k}})$ converging a.e. on $\Omega$ to $u_{0}$ and pointwise bounded by
a function belonging to $L^{2}$. Using growth condition (\ref{fwzrost0}) we
assert that%
\[%
%TCIMACRO{\dint \nolimits_{\Omega}}%
%BeginExpansion
{\displaystyle\int\nolimits_{\Omega}}
%EndExpansion
F(x,u_{n_{k}}(x)))dx\rightarrow%
%TCIMACRO{\dint \nolimits_{\Omega}}%
%BeginExpansion
{\displaystyle\int\nolimits_{\Omega}}
%EndExpansion
F(x,u_{0}(x)))dx.
\]
Assuming that the convergence
\[%
%TCIMACRO{\dint \nolimits_{\Omega}}%
%BeginExpansion
{\displaystyle\int\nolimits_{\Omega}}
%EndExpansion
F(x,u_{n}(x)))dx\rightarrow%
%TCIMACRO{\dint \nolimits_{\Omega}}%
%BeginExpansion
{\displaystyle\int\nolimits_{\Omega}}
%EndExpansion
F(x,u_{0}(x)))dx.
\]
does not hold and repeating the above reasoning we obtain a contradiction.
\end{proof}

Now, let us recall the following classical result: \textit{if }$E$\textit{ is
a reflexive Banach space and functional }$f:E\rightarrow R$\textit{ is weakly
sequentially lower semicontinuous and coercive, then there exists a global
minimum point of }$f$.

Thus, the functional $f$ given by (\ref{fdzial}) has a global minimum point
$u\in D(w((-\Delta)_{\omega}))$. Differentiability of $f$ means that $u$
satisfies (\ref{var}). Consequently, $u$ is a solution to
(\ref{problemgeneral}).

\begin{example}
Let us consider the Dirichlet problem for the equation%
\[
(\alpha_{k}(-\Delta)^{\beta_{k}}+...+\alpha_{0}(-\Delta)^{\beta_{0}}%
)^{2}u(x)=\frac{A}{2}\cos(x_{1}+...+x_{N})u^{2}(x)+b(x)\sin(u(x))
\]
in a bounded open set $\Omega\subset\mathbb{R}^{N}$ with $\alpha_{i}>0$ for
$i=0,...,k$ ($k\in\mathbb{N}\cup\{0\}$) and $0\leq\beta_{0}<\beta
_{1}<...<\beta_{k}$, where $0<A<\frac{\alpha_{k}^{2}}{M_{\beta_{k}}}$,
$M_{\beta_{k}}=\max\{\frac{1}{((\lambda_{j})^{\beta_{k}})^{2}};\ \lambda
_{j}<1\}$ (recall that if there is no $\lambda_{j}<1$, then $M_{\beta_{k}}%
=1$), $b\in L^{\infty}(\Omega,\mathbb{R})$. It is clear that the function
\[
F(x,u)=\frac{A}{2}\cos(x_{1}+...+x_{N})u^{2}+b(x)\sin u
\]
satisfies growth conditions (\ref{fwzrost0}, \ref{fwzrost1}, \ref{fwzrost2}).
Consequently, there exists a solution $u\in D(w^{2}((-\Delta)_{\omega
}))=D([(-\Delta)_{\omega}]^{2\beta_{k}})$ to the problem under consideration.
As we have shown, in the case of the domain $\Omega$ being of class $C^{1,1}$
or a bounded open convex polygon in $\mathbb{R}^{2}$, $(-\Delta)_{\omega}$ can
be replaced by $(-\Delta)$. If $\Omega=(0,\pi)\times(0,\pi)$ then the first
eigenvalue $\lambda_{1}$ of the operator $(-\Delta)_{\omega}=(-\Delta)$ is
equal to $2$ and consequently $M_{\beta_{k}}=1$. It is worth to point out that
Courant-Fisher theorems (see \cite[Theorems 8.4.1, 8.4.2]{Buttazzo}) can be
used to calculate the eigenvalues of $(-\Delta)_{\omega}$.
\end{example}


\begin{thebibliography}{99}                                                                                               %


\bibitem {Alex}A. Alexiewicz, Functional Analysis, PWN, Warsaw, 1969 (in Polish).

\bibitem {Buttazzo}H. Attouch, G. Buttazzo, G. Michaille, Variational Analysis
in Sobolev and BV Spaces. Applications to PDEs and Optimization, SIAM-MPS,
Philadelphia, 2006.

\bibitem {Barrios}B. Barrios, E. Colorado, A. de Pablo, U. S\'{a}nchez, On
some critical problems for the fractional Laplacian operator, Journal of
Differential Equations 252 (2012), 6133-6162.

\bibitem {Bors2}D. Bors, Stability of nonlinear Dirichlet BVPs governed by
fractional Laplacian, The Scientific World Journal 2014, Article ID 920537, 10
pages; DOI: 10.1155/2014/920537.

\bibitem {Brezis}H. Brezis, Functional Analysis, Sobolev Spaces and Partial
Differential Equations, Springer, New York, 2011.

\bibitem {CabreTan}X. Cabr\'{e}, J. Tan, Positive solutions of nonlinear
problems involving the square root of the Laplacian, Adv. Math. 224 (2010) 2052--2093.

\bibitem {Grisv1}P. Grisvard, Elliptic Problems in Nonsmooth Domains, Pitmann,
London, 1985.

\bibitem {Grisv2}P. Grisvard, Singularities in Boundary Value Problems,
Masson, Springer-Verlag, Paris, 1992

\bibitem {Helffer}B. Helffer, Spectral Theory and its applications, Cambridge,
United Kingdom, 2013.

\bibitem {Mlak}W. Mlak, An Introduction to the Hilbert Space Theory, PWN,
Warsaw, 1970 (in Polish).
\end{thebibliography}
\end{document}